\documentclass[12pt]{amsart}

\usepackage{xy, graphicx, color, hyperref}
 
\usepackage{amssymb}

\usepackage{tikz,verbatim,bbm}

\makeatletter
\def\subsection{\@startsection{subsection}{3}%
  \z@{.9\linespacing\@plus.7\linespacing}{.1\linespacing}%
  {\normalfont\bfseries}}
\makeatother

\makeatletter
\@namedef{subjclassname@1991}{Subject}
\makeatother

\title[Spinoriality for $\GL(n,q)$]{Spinoriality of  Orthogonal Representations of $\GL(n,q)$}
\author{Rohit Joshi}
\author{Steven Spallone}
 
\newtheorem{theorem}{Theorem}
\newtheorem{corollary}{Corollary}
\newtheorem{lemma}{Lemma}

\newtheorem{prop}{Proposition}

\theoremstyle{definition}
\newtheorem{remark}{Remark}

\newcommand{\nc}{\newcommand}

\nc{\thm}{\theorem}
\nc{\cor}{\corollary}
\nc{\mc}{\mathcal}
\nc{\mb}{\mathbb}
 
\nc{\mf}{\mathfrak}

\nc{\ul}{\underline}
\nc{\ol}{\overline}
\nc{\N}{\mb N}
\nc{\R}{\mb R}
\nc{\Z}{\mb Z}
\nc{\Q}{\mb Q}
\nc{\C}{\mb C}
\nc{\F}{\mb F}

\nc{\dmo}{\DeclareMathOperator}

\nc{\mat}[4]{
    \begin{pmatrix}
      #1 & #2 \\
      #3 & #4
    \end{pmatrix}
}

\dmo{\Ker}{Ker} \dmo{\val}{val} \dmo{\ord}{ord}
\dmo{\Pin}{Pin}
 
\dmo{\Int}{Int}
\dmo{\Or}{O}
\dmo{\SW}{SW}
\dmo{\odd}{odd}
\dmo{\sgn}{sgn}
\dmo{\trace}{trace}
\nc{\beq}{\begin{equation*}}
\nc{\eeq}{\end{equation*}}
\nc{\half}{\frac{1}{2}}
\nc{\vt}{a}
\dmo{\GSp}{GSp}

\dmo{\Mod}{mod}
\dmo{\ver}{ver}
\dmo{\diag}{diag}
\dmo{\res}{res}
\dmo{\lin}{lin}
 \dmo{\id}{id}
\dmo{\Sp}{Sp}
\dmo{\SO}{SO}
\dmo{\SL}{SL}
 \dmo{\Spin}{Spin}
\dmo{\GL}{GL}
 
\dmo{\ind}{ind}
 \dmo{\St}{St}
 \dmo{\st}{st}

\nc{\la}{\lambda}
  \nc{\eps}{\varepsilon}
  
 \nc{\lip}{\langle}
 \nc{\rip}{\rangle}
\nc{\gm}{\gamma}

\dmo{\Perm}{Perm}
\dmo{\Res}{Res}
\dmo{\Ind}{Ind}
\dmo{\tr}{tr}
\dmo{\Sym}{Sym}
\dmo{\reg}{reg}
\dmo{\End}{End}
\dmo{\Hom}{Hom}

\makeatletter
\def\Ddots{\mathinner{\mkern1mu\raise\p@
\vbox{\kern7\p@\hbox{.}}\mkern2mu
\raise4\p@\hbox{.}\mkern2mu\raise7\p@\hbox{.}\mkern1mu}}
\makeatother

\address{Bhaskaracharya Pratishthana 
	56/14, Erandavane, Damle Path,
Off Law College Road, Pune - 411 004,Maharashtra,India}
\email{rohitsj@students.iiserpune.ac.in}
\address{Indian Institute of Science Education and Research, Pune-411021,Maharashtra,India}
\email{sspallone@gmail.com}
\date{\today}

\begin{document}
\maketitle

\begin{abstract}

We determine which orthogonal representations $V$ of $\GL_n(\mb F_q)$ lift to the double cover $\Pin(V)$ of the orthogonal group $\Or(V)$. We cover all $n \geq 1$ and prime powers $q$, except  for $(n,q)=(3,4)$.
\end{abstract}

\tableofcontents 

\section{Introduction}
 
 Let $G$ be a finite group, and $\pi: G \to \Or(V)$ a complex orthogonal representation.  We say that $\pi$ is \emph{spinorial}, provided it lifts to the double cover $\Pin(V)$ of $\Or(V)$.   
 This is one of a series of papers investigating the spinoriality question for well-known groups.  Please see \cite{joshi} for criteria when $G$ is a connected reductive Lie group, \cite{GJ} for orthogonal groups, and \cite{Ganguly} for criteria when $G$ is a symmetric or alternating group.   This paper covers the group $G=\GL_n(\mb F_q)$, our foray into finite groups of Lie type.

In the case that $\det \pi=1$, the representation $\pi$ is spinorial iff the 2nd Stiefel-Whitney class of its real form vanishes. (See \cite{GKT}.) Thus the spinoriality of $\pi$ is equivalent to the existence of a spin structure for the vector bundle associated to $\pi$ over the classifying space  $BG$.   (See \cite[Section 2.6]{Benson}, \cite[Theorem II.1.7]{Lawson}.)  Determining spinoriality of Galois representations also has application in number theory:  see \cite{Serrespin}, \cite{Deligne}, and    \cite{DPDR}.

The group $G$ is a semidirect product of $H=\SL_n(\mb F_q)$ by the cyclic group $\GL_1(\mb F_q) \cong \mb F_q^\times$. For 
\beq
(n,q) \notin S=\{  (2,2), (2,3), (2,4), (3,2), (3,4), (4,2) \},
\eeq
$H$ is perfect and has no central extensions of even degree. Therefore every orthogonal representation of $H$ is spinorial, and as we show, $\pi$ is spinorial iff its restriction to 
$\GL_1(\mb F_q)$ is spinorial. The lifting problem for cyclic groups is not difficult.

To state our main theorem, let $a_1$ be the diagonal matrix in $\GL_n(\mb F_q)$ whose first diagonal entry is $-1$ and the rest are $1$'s.  Let $\Theta_\pi$ be the character of $\pi$,
and put
\beq
m_\pi=\frac{\Theta_\pi(1)-\Theta_\pi(a_1)}{2}.
\eeq

 \begin{thm} \label{main.theorem} Suppose $(n,q) \notin S$. An orthogonal representation $\pi$ of $\GL_n(\mb F_q)$ is spinorial iff:
 \begin{enumerate}
 \item $m_\pi \equiv 0 \mod 4$, when $q \equiv 1 \mod 4$
 \item $m_\pi \equiv 0 \text{ or } 3 \mod 4$, when $q \equiv 3 \mod 4$.
 \end{enumerate}
 Moreover when $q$ is even, all orthogonal representations of $\GL_n(\mb F_q)$ are spinorial.
 \end{thm}
\bigskip

The method adapts to similar families of finite groups of Lie type: to illustrate we produce a version of Theorem \ref{main.theorem} for $\GSp_{2n}(\mb F_q)$.

Using the well-known character table of $\GL_2(\mb F_q)$, we catalogue which irreducible orthogonal representations are spinorial. 
Also for $n \geq 3$, we deduce the following:
 
 \begin{thm} \label{our.exs.thm} Let $q$ be odd, $n \geq 3$, and $G=\GL_n(\mb F_q)$. Then all  irreducible orthogonal cuspidal representations, all orthogonal principal series representations, and the Steinberg representation of $G$ are spinorial. However there exist aspinorial orthogonal representations of $G$.   \end{thm}

\bigskip
 We also give criterion for the ``exceptional cases" when $(n,q) \in S$ for spinoriality in terms of character values, except when $(n,q) =(3,4)$.
 
 \bigskip
 
 The layout of this paper is  as follows. In Section \ref{prelim.section.rep} we set up notation, and review the theory of orthogonal representations and group cohomology.
  The spinoriality problem for cyclic groups is settled in Section \ref{cyclic.section}. Theorem \ref{main.theorem} is proved in Section \ref{dp.idea}, resulting from generalities about semidirect products. We also give the analogous result for $\GSp_{2n}(\mb F_q)$. Section \ref{n=2.section} is our enumeration of spinorial irreducible orthogonal representations of $\GL_2(\mb F_q)$ with $q$ odd. In Section \ref{exs.section} we demonstrate Theorem \ref{our.exs.thm} by using known character formulas. Finally in Section \ref{Exceptional.Section} we treat the cases of $(n,q) \in S$ but $(n,q) \neq (3,4)$.
 
 \bigskip
 
 {\bf Acknowledgements:} The first author was supported by a fellowship from Bhaskaracharya Pratishthana.  	We would also like to thank Sujeet Bhalerao, Neha Malik, and Jyotirmoy Ganguly for helpful conversations, and especially thank Dipendra Prasad for suggesting the approach of Section \ref{dp.idea}.
 
  \section{Notation and Preliminaries} \label{prelim.section.rep}
 
 Let $G$ be a finite group. All representations we consider are on finite-dimensional complex vector spaces. A ``linear character" is a one-dimensional representation.
 We denote the trivial linear character by `${\mathbf 1}$'.

    A linear character $\chi$ on $G$ is \emph{quadratic} when it takes values in $\{\pm 1\}$.  If $G$ is cyclic of even order, or if $G=\GL_n(\mb F_q)$, with $q$ odd, then there is a unique nontrivial quadratic linear character of $G$, denoted `$\sgn$' or `$\sgn_G$'. 
   In the cyclic case,  $\sgn(g)=-1$ when $g$ is a generator of $G$. In the general linear case, $\sgn_G=\sgn_{\mb F_q^\times} \circ \det$, where `$\det$' is the determinant.  
    
 If $H<G$ is a subgroup, and $\pi$ is a representation of $G$, we write `$\Res^G_H \pi$'  or `$\pi|_H$' for the restriction of $\pi$ to $H$. Write `$\Theta_\pi$' for the character of $\pi$.
 If $\sigma$ is a representation of $H$, write `$\Ind_H^G \sigma$' for the representation of $G$  induced from $\sigma$.
  
   Let $(\pi_1,V_1)$, $(\pi_2,V_2)$ be representations of finite groups $G_1,G_2$, respectively. We write $(\pi_1 \boxtimes \pi_2,V_1 \otimes V_2)$ for the external tensor product representation of $G_1 \times G_2$.

  \subsection{Orthogonal Representations} \label{rep.thy.ss}

A  representation  $(\pi,V)$ of $G$ is \emph{orthogonal}, provided it preserves a quadratic form on $V$, or equivalently a symmetric nondegenerate bilinear form on $V$.
In this case, $\pi$ can  be viewed as a homomorphism from $G$ to the   orthogonal group $\Or(V)$.
 
  Given a complex representation $(\pi,V)$, write $S(\pi)$ for the representation $\pi \oplus \pi^\vee$ on $V \oplus V^\vee$.
  Then $S(\pi)$ preserves the quadratic form  
\beq
\mc Q((v,v^*))=\lip v^*,v \rip,
\eeq
and is thus orthogonal.
 
 An orthogonal representation $\Pi$ of $G$ may be decomposed as
 \beq
 \Pi  \cong S(\pi) \oplus \bigoplus_j \varphi_j,
 \eeq
 where each $\varphi_j$ is irreducible orthogonal and $\pi$ does not have any irreducible orthogonal constituents.

 Let us say that a representation $\pi$ is \emph{orthogonally irreducible}, provided $\pi$ is orthogonal, and $\pi$ does not decompose into a direct sum of two \emph{orthogonal} representations. Thus, an orthogonal representation $\pi$ is orthogonally irreducible iff $\pi$ is irreducible, or of the form $S(\varphi)$ where $\varphi$ is irreducible but not orthogonal. We'll write `OIR' for ``orthogonally irreducible representation".

Note that a linear character is orthogonal iff it is quadratic.

 \subsection{The Pin Group}
 
Let $V$ be a finite-dimensional (complex) vector space, with a nondegenerate symmetric bilinear form $\Phi$.  The Clifford algebra $C(V)$ is the quotient of the tensor algebra $T(V)$ by the two-sided ideal generated by the set
\beq
 \{ v\otimes v +\Phi(v,v) : v\in V\}.
 \eeq
 It contains $V$ as a subspace.  Write $C(V)^\times$ for its group of units.   Then $\Pin(V)$ is the subgroup of $C(V)^\times$ generated by the unit vectors in $V \subset C(V)$.  
 There is a unique homomorphism  $\rho: \Pin(V) \to \Or(V)$ taking a unit vector $u$ to the reflection of $V$ determined by $u$.  Note that $-1 =u^2 \in \Pin(V)$ and $\rho(-1)=1$. Since $\Or(V)$ is generated by reflections, $\rho$ is surjective; it is a nontrivial double cover of $\Or(V)$.    See \cite[Chapter 20]{Fulton.Harris} for details.

 Suppose $\pi: G \to \Or(V)$ is an orthogonal representation, where $\Or(V)$ is determined by $\Phi$. We say that $\pi$ is \emph{spinorial}, provided there is a homomorphism $\hat \pi: G \to \Pin(V)$  so that $\rho \circ \hat \pi=\pi$.  In this situation, $\hat \pi$ is called a \emph{lift} of $\pi$.
 
 \begin{lemma} \label{A,B.lemma} Let $A \in \Or(V)$ with $A^2=1$. Let $m$ be the multiplicity of $-1$ as an eigenvalue of $A$. Suppose $B \in \Pin(V)$ with $\rho(B)=A$. Then $B^2=1$ iff $m \equiv 0,3 \mod 4$.
 \end{lemma}
 
 \begin{proof}
   Let $e_1, \ldots, e_{m_\pi}$ be an orthonormal basis of the $-1$-eigenspace of $A$, so that $A$ is the product of the reflections in each $e_j$.
Therefore  $B= \pm e_1 \cdots e_{m_\pi} \in \Pin(V)$.   
One computes
 \beq
 B^2=(-1)^{\half m(m +1)},
 \eeq
 and this exponent is even iff $m$ is congruent to $0$ or $3$ modulo $4$.
 \end{proof}
 
 Note that here
 \begin{equation} \label{m.formula}
 m=\frac{\dim V-\trace(A)}{2}.
 \end{equation}
 
 \subsection{Extensions and Cohomology}

Recall \cite[Section 6.6]{CW} that to an extension 
\beq
1 \to A \to E \to G \to 1
\eeq
of a group $G$ by an abelian group $A$ corresponds a cohomology class $c_E \in H^2(G,A)$. Moreover, $E$ is a split extension iff $c_E=0$.
The extension  $\rho:\Pin(V) \to \Or(V)$ does not split, and has fibre $A=\Z/2\Z$.

A homomorphism $\varphi: G' \to G$ induces a  \emph{pullback}  extension
\beq
1 \to A \to E' \to G' \to 1,
\eeq
where $E'=E \times_G G'$ and $E' \to G'$ is projection to the second component.
Moreover $c_{E'}$ is the image of $c_E$ under the induced map $\varphi^*: H^2(G,A) \to H^2(G',A)$.

If $(\pi,V)$ is an orthogonal representation of a group $G$, then the pullback extension is split iff $\pi$ is spinorial.

\section{Abelian Groups} \label{cyclic.section}

\subsection{Cyclic Groups}

Suppose first that $n$ is odd. If $\pi$ is an orthogonal representation of $C_n$, then the pullback extension of $C_n$ is split by the Schur-Zassenhaus Theorem.
Therefore $\pi$ is spinorial.

Now let $n$ be even. Write $C_n$ for the cyclic group of order $n$. Fix a generator $g$ of $C_n$.   
 We say a linear character $\chi$ is \emph{even} provided $\chi(g^{n/2})=1$, and that it is \emph{odd} when $\chi(g^{n/2})=-1$.

It is well-known \cite[Section 6.2]{CW} that $H^2(C_n,\Z/2\Z)$ has only one nonzero element; it corresponds to the nonsplit extension
\begin{equation} \label{cn.ext}
1 \to C_2 \to C_{2n} \to C_n \to 1.
\end{equation}

\begin{prop} \label{d.to.n} Suppose $d$ is even, and $n$ is a multiple of $d$. Then the restriction map $H^2(C_n,\Z/2\Z) \to H^2(C_d,\Z/2\Z)$ is an isomorphism. An orthogonal representation $\pi$ of $C_n$ is spinorial iff its restriction to $C_d$ is spinorial.
\end{prop}
 
\begin{proof}
It is enough to show that the pullback of \eqref{cn.ext} to $C_d < C_n$ does not split. But this pullback is 
\beq
1 \to C_2 \to C_{2d} \to C_d \to 1.
\eeq
This gives the first statement, and the last statement follows.
\end{proof}

Let $\pi$ be an orthogonal representation of $C_n$ with $n$ even. Let $m_\pi$ be the multiplicity of $-1$ as an eigenvalue of $\pi(g^{n/2})$. By \eqref{m.formula}, we have
\beq
m_\pi= \frac{\Theta_{\pi}(1) - \Theta_{\pi}(g^{n/2})}{2}.
\eeq

  \begin{prop} With notation as above, the representation $\pi$ is spinorial iff $m_\pi \equiv 0,3 \mod 4$. 
  \end{prop}
  
  \begin{proof} 
  By Proposition \ref{d.to.n}, we may assume $n=2$.  The proposition then follows from applying Lemma \ref{A,B.lemma} to $A=\pi(g)$.
 \end{proof}
 
When $n$ is a multiple of $4$, the integer $m_\pi$ is twice the multiplicity of $i$ (or of $-i$) as an eigenvalue of $\pi(g^{\frac{n}{4}})$.
 In particular, $m_\pi$ is even. Therefore in this case, $\pi$ is spinorial iff $m_\pi$ is a multiple of $4$.

  \bigskip

To summarize:

 \begin{prop} \label{chat} Let $\pi$ be an orthogonal representation of $C_n$. For $n$ even, let $m_\pi$ be the multiplicity of $-1$ as an eigenvalue of $\pi(g^{n/2})$.
 \begin{enumerate}
 \item If $n$ is odd, then $\pi$ is spinorial.
 \item If $n \equiv 0 \mod 4$, then $\pi$ is spinorial iff $m_\pi \equiv 0 \mod 4$.
  \item If $n \equiv 2 \mod 4$, then $\pi$ is spinorial iff $m_\pi \equiv 0$ or $3 \mod 4$.
 \end{enumerate}
 \end{prop}

 \subsection{Elementary Abelian $2$-groups}

  Let $E$ be a rank $d$ elementary $2$-group, i.e., $E \cong C_2 ^d$.
 Then $E$ has the following presentation  (via generators and relations):
   \beq
  E= \lip e_1, \ldots, e_d \mid e_i^2, (e_i e_j)^2, 1 \leq  i,j \leq d \rip.
   \eeq

\begin{prop} \label{elementary} Let $\pi$ be an orthogonal representation of $E$. Then $\pi$ is spinorial iff $\forall e \in E$, the integer
\beq
m_e=\frac{\Theta_\pi(1)-\Theta_\pi(e)}{2}
\eeq
is congruent to $0$ or $3$ mod $4$.
\end{prop}

\begin{proof} From the presentation, $\pi$ is spinorial iff the lift of each $\pi(e_i)$ and $\pi(e_i e_j)$ in $\Pin(V)$ squares to $1$. 
If $\pi$ is spinorial, then the lift of each $\pi(e)$ squares to $1$. The result then follows from Lemma \ref{A,B.lemma}.
\end{proof}

\section{Main Theorem} \label{dp.idea}

The group $G=\GL_n(\mb F_q)$ is the semidirect product of $H=\SL_n(\mb F_q)$ with the cyclic group $\GL_1(\mb F_q)$.  In this section we show that, except for finitely many $(n,q)$, an orthogonal representation of $G$ is spinorial iff its restriction to this cyclic group is spinorial. Our main theorem will then follow from the previous section.

\subsection{Semidirect Products}

  Let $G$ be a group, and $\pi$ a spinorial (orthogonal) representation of $G$. The group $\Hom(G,\{\pm 1\})$ of quadratic linear characters $\chi$ acts on the set of lifts $\hat \pi$ of $\pi$ via
  \beq
  (\chi \odot \hat \pi)(g)=\chi(g) \hat \pi(g).
  \eeq
   This defines a simply transitive action of the group of orthogonal linear characters on the set of lifts of $\pi$. In particular:
  \begin{lemma} \label{OLCs} 
  If $G$ has no subgroups of index $2$, then the lift of $\pi$ is unique.
  \end{lemma}

 \begin{prop} \label{H,K} Let $G$ be a finite group, and $H,K$ subgroups of $G$ so that $H$ is normal in $G$, $H \cap K=\{1\}$, and $G=HK$. 
 Suppose further that $H$ has no subgroups of index $2$. If $\pi$ is an orthogonal representation of $G$, then $\pi$ is spinorial iff its restrictions to both $H$ and $K$ are spinorial.
 \end{prop}
 
 \begin{proof} Suppose $\pi|_H$ and $\pi|_K$ are spinorial, let $\hat \pi_H$ and $\hat \pi_K$ be the respective lifts. 
  
{\bf Claim:} Given $k_0 \in K$, we have
  \begin{equation} \label{int.eqn}
  \Int(\hat \pi_K(k_0)) \circ \hat \pi_H \circ \Int(k_0)^{-1}= \hat \pi_H.
  \end{equation}
  To see this, note that the LHS evaluated at $h \in H$ equals
\beq
  \hat \pi_K(k_0) \hat \pi_H(k_0^{-1}hk_0) \hat \pi_K(k_0^{-1}),
\eeq
  and applying $\rho$ gives $\pi(h)$. The claim then follows from Lemma \ref{OLCs}.
 
 Next, define $\hat \pi(g)=\hat \pi_H(h) \hat \pi_K(k)$ when $g=hk$ with $h \in H$ and $k \in K$. 
 Using \eqref{int.eqn}, it is straightforward to check that $\hat \pi$ is a homomorphism, and indeed a lift of $\pi$. \end{proof}
 
 \subsection{Case of Odd Schur Multiplier}
 
 Recall from \cite[Section 6.9]{CW} that a perfect group $G$ has a universal central extension $\beta: \tilde G \to G$, in the sense that it factors through every central extension of $G$. Let $M(G)=\ker \beta$; this abelian group is the \emph{Schur multiplier} of $G$. It is isomorphic to the abelian group $H^2(G,\C^\times)$.
 
 \begin{prop} \label{schur.mult.prop}  Let $G$ be a perfect group with $|M(G)|$ odd. Then every orthogonal representation $\pi$ of $G$ is spinorial.
 \end{prop}
 
 \begin{proof} Let $\beta: \tilde G \to G$ be the universal central extension of $G$.
 
 The extension of $G$ associated to $\pi$ is a degree $2$ central extension $\rho': G' \to G$. By the universality of $\tilde G$, there exists a morphism $\alpha: \tilde G \to G'$ of covers, i.e., so that $\rho' \circ \alpha=\beta$.
 But $\alpha(\ker \beta)$ is in the kernel of $\rho'$, which has order $2$. But since $|\ker \beta|$ is odd, it must be that $\alpha(\ker \beta)=1$, so $\alpha$ factors to  $\alpha': G \to G'$, a splitting of $\rho'$. The composition of $\alpha'$ with the projection of $G'$ to $\Pin(V)$ is a lift of $\pi$.  \end{proof}
 
 \subsection{Application to $\GL_n(\mb F_q)$}
 
 Let $S=\{  (2,2), (2,3), (2,4), (3,2), (3,4), (4,2) \}$. It is well-known that if $(n,q) \notin S$, then  $\SL_n(\mb F_q)$ is perfect, and $M(\SL_n(\mb F_q))$ has odd order. In what follows we regard $\GL_1(\mb F_q)$ as a subgroup of $\GL_n(\mb F_q)$ in the usual way, via the top left entry.
   
 \begin{cor} \label{ankh} Let $G=\GL_n(\mb F_q)$, and suppose $(n,q) \notin S$. Then an orthogonal representation $\pi$ of $G$ is spinorial iff its restriction to $\GL_1(\mb F_q)$ is spinorial.
 \end{cor}
 
 \begin{proof} Take $H=\SL_n(\mb F_q)$ and $K=\GL_1(\mb F_q)$. Then $H,K$ satisfy the conditions of Proposition \ref{H,K}. By Proposition \ref{schur.mult.prop}, the restriction of $\pi$ to $H$ is spinorial.
 \end{proof}

\begin{remark} This method was suggested by D. Prasad.
\end{remark}

\begin{prop} Let $G=\GL_n(\mb F_q)$, with  $(n,q) \notin S$. Then 
\beq
H^2(G,\Z/2\Z) \cong H^2(\GL_1(\mb F_q),\Z/2\Z).
\eeq
 Thus when $q$ is odd, there is only one nontrivial extension of $G$ by $\Z/2\Z$.
\end{prop}

\begin{proof} Again let $H=\SL_n(\mb F_q)$. Since $H$ is perfect, we have $H^1(H,\mb C^\times)=H^1(H,\Z/2\Z)=0$. From the exact sequence of the squaring map,
\beq
1 \to \mu_2 \to \C^\times \to \C^\times \to 1,
\eeq
 we deduce an injection $H^2(H,\Z/2\Z) \hookrightarrow H^2(H,\mb C^\times)$. But since $H^2(H,\mb C^\times) \cong M(H)$, which has odd order, it must be that $H^2(H,\Z/2\Z)=0$. According to \cite[Chapter 7, Section 6, Corollary]{Serre.Local}, inflation gives an isomorphism
\beq
H^2(G/H,\Z/2\Z) \cong H^2(G,\Z/2\Z).
\eeq
The proposition follows since $G/H \cong \GL_1(\mb F_q)$.
\end{proof}

   Let $a_1=\diag(-1,1,1,\ldots) \in G$.
   
 \begin{thm} \label{gl.thm} Let $G=\GL_n(\mb F_q)$, with $(n,q) \notin S$, and $\pi$ an orthogonal representation of $G$.  Put 
 \beq
 m_\pi=\frac{\Theta_\pi(1)-\Theta_\pi(a_1)}{2}.
 \eeq
 \begin{enumerate}
 \item If $q$ is even, then $\pi$ is spinorial.
  \item If $q \equiv 1 \mod 4$, then $\pi$ is spinorial iff $m_\pi \equiv 0 \mod 4$.
 \item If $q \equiv 3 \mod 4$, then $\pi$ is spinorial iff  $m_\pi \equiv 0$ or $3 \mod 4$.
 \end{enumerate}  
 \end{thm}
 
 \begin{proof} Note that if $g$ is a generator of $\mb F_q^\times$, with $q$ odd, then $g^{(q-1)/2}=-1$, as it is the unique element of order $2$. 
 The Theorem then follows from Corollary \ref{ankh} and Proposition \ref{chat}.
 \end{proof}

\subsection{Permutation Matrices Detect Spinoriality}

Consider the symmetric group $S_n$  as a subgroup of $G$ via permutation matrices. 
 \begin{prop}  An orthogonal representation $\pi$ of $G$ as above is spinorial iff its restriction to $S_n$ is spinorial.
 \end{prop}
  
\begin{proof} This is clear if $q$ is even, as every representation of $G$ is spinorial. So suppose $q$ is odd. Clearly if $\pi$ is spinorial, then $\pi|_{S_n}$ is spinorial. Suppose $\pi|_{S_n}$ is spinorial. Any transposition in $S_n$ is conjugate in $G$ to $a_1$.
According to Theorem 1.1 of  \cite{Ganguly}, we have $m_\pi \equiv 0$ or $3 \mod 4$. Therefore $\pi$ is spinorial.
\end{proof}
 
 \subsection{Other Finite Groups of Lie Type}
 
The proof of Theorem \ref{gl.thm} adapts to other finite groups of Lie type. For instance, let $G=\GSp_{2n}(\mb F_q)$. To fix ideas say 
\beq
J=\left(\begin{array}{ccccc}
&&&&1\\
&&&1&  \\
&&\Ddots&&\\
&-1&&&\\
-1&&&&
\end{array}\right),
\eeq
and 
\beq
G=\{ g \in \GL_{2n}(\mb F_q) \mid \exists \lambda \in \mb F_q^\times \text{ so that } {}^tg Jg= \lambda J\}.
\eeq 
Then $G$ is the semidirect product of $H=\Sp_{2n}(\mb F_q)$ with $\GL_1(\mb F_q)$. Barring finitely many $(n,q)$, $H$ is perfect with Schur multiplier $1$.
Thus Theorem \ref{gl.thm} holds for such $(n,q)$ with  
\beq
a_1=\left(\begin{array}{ccccc}
-1&&&&\\
&-1&&&  \\
&&\ddots&&\\
&&&1&\\
&&&&1
\end{array}\right),
\eeq
with the entry `$-1$' repeated $n$ times followed by $n$ `$1$'s.

  \section{$\GL_2(\mb F_q)$} \label{n=2.section}
 
  Let $G=\GL_2(\mb F_q)$, with $q$ odd. We write  $A$  for the subgroup of diagonal matrices, $U$ for the upper unitriangular subgroup, and $B=AU$. Write $Z$ for the center of $G$.    Choosing an $\mb F_q$-basis of $\mb F_{q^2}$ gives an elliptic torus $T< G$, which we identify with $\mb F_{q^2}^\times$.
 
 \subsection{Catalogue of Irreducible Representations of $G$}

   We follow the notation of  \cite[Chapter 2]{Bushnell.Henniart} to enumerate the irreducible representations of $G=\GL_2(\mb F_q)$.

  Let $\St_G$ be the  Steinberg representation, so that $\Ind_B^G {\bold 1}={\bold 1} \oplus \St_G$. For $\chi, \chi'$ linear characters of $\mb F_q^\times$, let $\pi(\chi,\chi')$ be the  parabolic induction $\Ind_B^G (\chi \boxtimes \chi')$. 
  
  If $\theta$ is a character of $\mb F_{q^2}^\times$, write $\theta^\tau$ for its composition with the nontrivial element $\tau$ of the Galois group of $\mb F_{q^2}$ over $\mb F_{q}$.  We say that $\theta$ is \emph{regular}, provided $\theta \neq \theta^\tau$.
   Fix a nontrivial linear character $\psi$  of $U$, and define a linear character of $ZU$ by $\theta_{\psi}(zu) = \theta(z) \psi(u)$. 
  For $\theta$ regular, put
  \beq
 \pi_\theta =\Ind_{ZU}^G \theta_\psi - \Ind_{T}^G \theta.
 \eeq 
 These are the cuspidal representations.

  The   irreducible representations of $G$ are as follows:
  
  \begin{enumerate}
  \item The linear characters,
  \item The principal series representations $\pi(\chi,\chi')$, with $\chi \neq \chi'$  linear characters of $\mb F_q^\times$,
  \item Twists $ \St_G \otimes \chi$ of the Steinberg for a linear character $\chi$ of $G$,
  \item The cuspidal representations $\pi_\theta$, with $\theta$ a  regular character of  $T$.
  \end{enumerate}

    The irreducible orthogonal representations of $G$ are:
  \begin{enumerate}
  \item ${\bold 1}$ and $\sgn_G$,
  \item  $\pi({\bold 1},\sgn)$,
  \item  $\pi(\chi,\chi^{-1})$ with $\chi$ not quadratic,
  \item $\St_G$ and $\St_G \otimes \sgn_G$,
  \item  $\pi_\theta$, where $\theta^\tau =\theta^{-1}$.
  \end{enumerate}
   
Recall that the OIRs of $G$ are either irreducible orthogonal $\pi$, or $S(\pi)$ for $\pi$ irreducible but not orthogonal.

\subsection{List of spinorial OIRS for $G$}

 We may immediately enumerate  the spinorial OIRs  of $G=\GL_2(\mb F_q)$ by invoking Theorem \ref{gl.thm}, since the character table of $G$ is well-known.
 
\begin{thm}
	The following is the complete list of  spinorial OIRs of $G$:
	\begin{enumerate}
		\item Case $q \equiv 1 \pmod 8$
		\begin{itemize}
			\item ${\bold 1}, \sgn_G$
			\item $\pi(\chi,\chi^{-1})$ with $\chi$ even
			\item  $\St_G$ and $\St_G \otimes \sgn_G$
			\item $\pi( {\bold 1},\sgn)$
			\item All cuspidal OIRs		
			\item $S(\chi)$ with $\chi$ even and $\chi^2 \neq {\bold 1}$
			\item $S(\St_G \otimes \chi) $ with $\chi$  even and $\chi^2 \neq {\bold 1}$
			\item  $S(\pi(\chi_1,\chi_2)) $ with $\chi_1 \cdot \chi_2$ even and $\chi_i^2 \neq {\bold 1}$	
			\item  $S(\pi_\theta)$, with $\theta^\tau \neq \theta^{-1}$
		\end{itemize}
		
		\item Case $q \equiv 3 \pmod 8$
		
		\begin{itemize}
			\item ${\bold 1}$
			\item $\pi(\chi,\chi^{-1})$ with $\chi$ odd
			\item $S(\chi)$ with $\chi$ even and $\chi^2 \neq{\bold 1}$
			\item $S(\St_G \otimes \chi)$ with $\chi$  odd and $\chi^2 \neq {\bold 1}$
			\item  $S(\pi(\chi_1,\chi_2))$ with $\chi_1 \cdot \chi_2$ odd with $\chi_i^2 \neq {\bold 1}$	
		\end{itemize}

		\item Case $q \equiv 5 \pmod 8$	
		
		\begin{itemize}
			\item ${\bold 1}, \sgn_G$
			\item $\pi(\chi,\chi^{-1})$ with $\chi$ odd
			\item $S(\chi) $ with $\chi$ even and $\chi^2 \neq {\bold 1}$
			\item $S(\St_G \otimes \chi)$ with $\chi$  even and $\chi^2 \neq{\bold 1}$
			\item  $S(\pi(\chi_1,\chi_2))$ with $\chi_1 \cdot \chi_2$ even and $\chi_i^2 \neq {\bold 1}$	
			\item  $S(\pi_\theta)$, with $\theta^\tau \neq \theta^{-1}$
		\end{itemize}

		\item Case $q \equiv 7 \pmod 8 $	 	
		
		\begin{itemize}
			\item ${\bold 1}$
			\item $\pi(\chi,\chi^{-1})$ with $\chi$ even
			\item  $\St_G$ and $\St_G \otimes \sgn_G$
			\item $\pi( {\bold 1},\sgn)$
			\item All cuspidal irreducible orthogonal representations
			\item $S(\chi)$ with $\chi$ even and $\chi^2 \neq {\bold 1}$
			\item $S(\St_G \otimes \chi)$ with $\chi$  odd and $\chi^2 \neq {\bold 1}$
			\item  $S(\pi(\chi_1,\chi_2))$ with $\chi_1 \cdot \chi_2$ odd and $\chi_i^2 \neq{\bold 1}$	
		\end{itemize}
 		
	\end{enumerate}
\end{thm}

 \section{Case of $n \geq 3$} \label{exs.section}
 
 In this section, we apply Theorem \ref{gl.thm} to list all spinorial OIRs in familiar situations for $n \geq 3$.
We assume  that $q=p^r$ for some odd prime $p$.  
\subsection{Existence of Aspinorial Representations}

First,  note that aspinorial orthogonal representations always exist:

 \begin{prop} For each $n,q$ with $q$ odd,  there exist aspinorial orthogonal representations of $\GL_n(\mb F_q)$. 
 \end{prop}
 
 \begin{proof}
Take an odd character $\chi_0: \mb F_q^\times \to \C^\times$, and put $\pi=S(\chi_0 \circ \det_G)$. Then $\Theta_\pi(1)=2$ and $\Theta_\pi(a_1)=-2$, so $m_\pi=2$, hence 
$\pi$ is aspinorial by Theorem \ref{gl.thm}.
   \end{proof}

\begin{remark} Also, $\sgn_G$ is spinorial iff $q \equiv 1 \mod 4$.
\end{remark}

 \subsection{Cuspidal Representations}
   Fixing an $\mb F_q$-basis of $\mb F_{q^n}$ allows us to identify  $\mb F_{q^n}^\times$ with a subgroup $T< G$.  In particular, $T$ gets an action of the Galois group of $\mb F_{q^n}$ over $\mb F_q$. Let $\theta$ be a \emph{regular} linear character of $T$, meaning it is unequal to any linear character in its Galois orbit. Associated to $(T,\theta)$ is an irreducible representation $\pi_{\theta}$ whose character is denoted $R_{T,\theta}$ in \cite{carter}. Such representations are called \emph{cuspidal}. The character of $\pi_{\theta}$ is supported on elements of $G$ conjugate to elements of $T$. Thus if $n \geq 3$, then $\Theta_{\pi_\theta}$ vanishes on $\vt_1$. Its degree is 
    $\psi_{n-1}(q)$, where generally
 \beq
 \psi_m(q)=\prod_{i=1}^m (q^i-1).
 \eeq
  It is clear that $\psi_m(q) \equiv 0 \mod 8$ for $m \geq 2$. Therefore for $n \geq 3$, we have $m_\pi \equiv 0 \mod 8$. We deduce that every orthogonal cuspidal representation of $G$ is spinorial.
 
 \subsection{Steinberg Representation}
 
 Let $\pi$ be the Steinberg representation, and $a \in A$.  Let `$Z_G(a)$' denote  the centralizer of $a$ in $G$. According to \cite[Thm 6.4.7]{carter}, we have
\beq
\Theta_\pi(a)=p^{k},
\eeq
where $p^k$ is the highest power of $p$ dividing $|Z_G(a)|$.
Thus
\beq
\Theta_{\pi}(1)=q^{\half n(n-1)},
\eeq
and
\beq
\Theta_{\pi}(\vt_1)=q^{\half (n-1)(n-2)}.
 \eeq

From this we deduce:
 
 \begin{prop} The Steinberg representation of $\GL_n(\mb F_q)$ is spinorial whenever $n \geq 3$.
 \end{prop}

  \subsection{Principal Series}
 
 Let $\chi_1, \ldots, \chi_n$ be linear characters of $\mb F_q^\times$, and let $B$ be the subgroup of upper triangular members of $G$. Put
 \begin{equation} \label{princ.series.here}
 \pi=\Ind_{B}^{G} \left(\chi_1 \boxtimes \cdots \boxtimes \chi_n \right).
\end{equation}
Such representations are called \emph{principal series representations} of $G$. Then $\pi$ is orthogonal iff
\begin{equation} \label{ind.orth.crit}
\{\chi_1, \ldots, \chi_n \}=\{ \chi_1^{-1}, \ldots, \chi_n^{-1} \}
\end{equation}
as multisets.

  Recall the $q$-factorial:
  \beq
\begin{split}
[n]_q ! &=\prod_{i=1}^{n}\frac{q^i-1}{q-1} \\
&= (q+1)(q^2+q+1) \cdots (q^{n-1}+ \cdots + 1).\\
\end{split}
\eeq

  \begin{thm} If $n \geq 3$, then all orthogonal principal series of $\GL_n(\mb F_q)$ are spinorial.
 \end{thm}

 \begin{proof}
 Let $\pi$ be as in \eqref{princ.series.here}.  The relevant character values of $\pi$ are
\beq
\Theta_\pi(1)=[n]_q!
\eeq
and
\beq
\Theta_\pi(\vt_1)=[n-1]_q! \cdot \sum_{i=1}^n \chi_i(-1).
\eeq
 
These can be inferred from \cite[Section 2]{Green}, \cite[Chapter IV, Section 3]{Macdonald}, or \cite[Proposition 7.5.3]{carter}.

Note that $[n]_q!$ is divisible by $8$ whenever $n \geq 4$. Therefore both $\Theta_\pi(1)$ and $\Theta_{\pi}(\vt_1)$ are divisible by $8$ unless $n \leq 4$. Thus we need only consider the cases of $n=3,4$.
 
 \bigskip
 
 {\bf Case $n=3$:} By  \eqref{ind.orth.crit}, there is a linear character $\chi$ of $\mb F_q^\times$ so that 
 \beq
 \{ \chi_1,\chi_2,\chi_3 \}= \{ {\bold 1},\chi,\chi^{-1}\} \text{ or }  \{ \sgn,\chi,\chi^{-1}\}.
 \eeq
For $q \equiv 1 \mod 4$, we have $\sgn(-1)=1$, so
 \beq
 \begin{split}
\Theta_\pi(1) -  \Theta_{\pi}(\vt_1) &= (q^2+q+1)(q+1)-(q+1)(1 \pm 2) \\
 			&= (q+1)(q^2+q \pm 2), \\
			\end{split}
			\eeq
			which is evidently a multiple of $8$. For $q \equiv 3 \mod 4$, it's clear that $\Theta_\pi(1) -  \Theta_{\pi}(\vt_1)$ is an even multiple of $q+1$, and hence is also a multiple of $8$. Thus $m_\pi$  is divisible by $4$, and we conclude for $n=3$ that $\pi$ is spinorial.

 \bigskip
 
 {\bf Case $n=4$:} Here $\Theta_{\pi}(1)=[4]_q$ is a multiple of $8$, so it is enough to show the same for $\Theta_{\pi}(a_1)$.
 By \eqref{ind.orth.crit}, either there is a  linear character $\chi$  so that
 \beq
 \{ \chi_1,\chi_2,\chi_3,\chi_4 \}=\{ {\bold 1},\sgn, \chi,\chi^{-1} \},
 \eeq
 or there are linear characters $\chi_A,\chi_B$ so that 
  \beq
 \{ \chi_1,\chi_2,\chi_3,\chi_4 \}=\{  \chi_A,\chi_A^{-1} , \chi_B,\chi_B^{-1}\}.
 \eeq
 Under the first possibility,
 \beq
  \Theta_{\pi}(\vt_1)=(q^2+q+1)(q+1)(2+2\chi(-1)),
  \eeq
  which is evidently divisible by $8$.   The second possibility is similar, and this completes the proof.   \end{proof}

 \section{Exceptional Cases} \label{Exceptional.Section}

In conclusion we treat spinoriality for the groups $G=\GL_n(\mb F_q)$ with $(n,q) \in S$, except $(n,q)=(3,4)$.  
 In each of the cases below, let $\pi$ be an orthogonal representation of $G$.  We give a criterion in terms of character values of $\pi$ to determine its spinoriality. Unfortunately for $G=\GL_3(\mb F_4)$ we do not currently have the tools to settle this question.
 
 \subsection{$G=\GL_2(\mb F_2)$}
 
 Let $K$ be the subgroup generated by $\mat 1101$, and $H$ the subgroup generated by $\mat 1110$. Then $|H|=3$, $|K|=2$ , and 
 the hypotheses of Proposition \ref{H,K} hold.  It follows that $\pi$   is spinorial iff $m_\pi \equiv 0,3 \mod 4$, where
 \beq
 m_\pi=\frac{\Theta_\pi(1)-\Theta_\pi \left(\mat 1101 \right)}{2}.
 \eeq

 \subsection{$G=\GL_2(\mb F_3)$}
Let $A< G$ be the subgroup of diagonal members; it is a Klein $4$-group. According to \cite[Section 8]{Quillen}, the restriction map 
\beq
H^2(G,\Z/2\Z) \to H^2(A,\Z/2\Z)
\eeq
 is injective. Therefore  $\pi$ is spinorial iff its restriction to $A$ is spinorial. 
 
 Put  
 \beq
 m_\pi=\frac{\Theta_\pi(1)-\Theta_\pi \left(\mat {-1}001 \right)}{2}
 \eeq
 and 
 \beq
 m_\pi'=\frac{\Theta_\pi(1)-\Theta_\pi\left(\mat {-1}00{-1} \right)}{2}.
 \eeq
 It follows from Proposition \ref{elementary} that $\pi$ is spinorial iff $m_\pi,m_\pi' \equiv 0$ or $3 \mod 4$.

  \subsection{$G=\GL_2(\mb F_4)$}
  
  In this case $G$ is the direct product of $H=\SL_2(\mb F_4)$ and the center $K$, which is cyclic of order $3$. From \cite[3.10]{Wilson}, we may identify $H$ with the alternating group $A_5$.  By Proposition \ref{H,K}, $\pi$ is spinorial iff $\pi|_{A_5}$   is spinorial. 
We conclude from  \cite[Theorem 1.1]{Ganguly} that $\pi$ is spinorial iff $\Theta_\pi(1) \equiv \Theta_\pi((12)(34)) \mod 8$. Here we are using the usual cycle notation for permutations.
  
 \subsection{$G=\GL_3(\mb F_2)$} 
   This is the simple group of order 168.    The upper triangular subgroup $U$ of $G$ is a $2$-Sylow subgroup. According to   \cite[Corollary 5.2, Chapter II]{Milgram}, the
   restriction map
   \beq
   H^2(G,\Z/2\Z) \to H^2(U,\Z/2\Z)
   \eeq
    is an injection, and therefore $\pi$ is spinorial iff its restriction to $U$ is spinorial. The group $U$ is dihedral of order $8$. According to \cite[Proposition 3.3, page 323]{fied}, there are subgroups $E_1$, $E_2$, both Klein $4$-groups, with the property that the sum of the restriction maps,
    \beq
    H^2(U,\Z/2\Z) \to H^2(E_1,\Z/2\Z) \oplus H^2(E_2,\Z/2\Z),
    \eeq
 is injective. It follows that $\pi$ is spinorial iff its restrictions to $E_1$ and $E_2$ are both spinorial. Therefore $\pi$ is spinorial iff $\pi(u)$ has lift squaring to $1$ for each $u \in U$ of order $2$. There are two $G$-conjugacy classes of order $2$ elements in $U$, represented by
 \beq
 u_1=    \begin{pmatrix}
      1 & 1 & 0 \\
       & 1 & 0 \\
       &  & 1 \\
    \end{pmatrix}
 \eeq
 and
  \beq
 u_2=    \begin{pmatrix}
      1 & 1 & 1 \\
       & 1 & 0 \\
       &  & 1 \\
    \end{pmatrix}.
    \eeq
    
 For $i=1,2$, let 
 \beq
 m_{\pi,i}=\frac{\Theta_\pi(1)-\Theta_\pi(u_i)}{2}.
 \eeq
 
 From the above and Lemma \ref{A,B.lemma}, we deduce:
 
 \begin{prop} The representation $\pi$ is spinorial iff $m_{\pi,1},m_{\pi,2} \equiv 0,3 \mod 4$.
 \end{prop}

 \subsection{$G=\GL_4(\mb F_2)$}
From \cite[3.10]{Wilson}, $G$ is isomorphic to $A_8$. Again by \cite[Theorem 1.1]{Ganguly}, $\pi$ is spinorial iff $\Theta_\pi(1) \equiv \Theta_\pi((12)(34)) \mod 8$. 

 \bibliographystyle{plain}
 
\bibliography{refs}

 \end{document}